\documentclass[12pt]{amsart}
\usepackage{fancyhdr}
\usepackage{stmaryrd}
\usepackage{calrsfs}

\usepackage{amsmath}
\usepackage[nospace,noadjust]{cite}
\usepackage{amsfonts}
\usepackage{amssymb,enumerate}
\usepackage{amsthm}
\usepackage{cite}
\usepackage{comment}
\usepackage{color}
\usepackage[all]{xy}
\usepackage{hyperref}
\usepackage{lineno}
\usepackage{stmaryrd}

\usepackage{mathtools}
\usepackage{bm}
\usepackage{graphicx}
\usepackage{psfrag}
\usepackage{url}

\usepackage{fancyhdr}
\usepackage{tikz-cd}

\usepackage{dirtytalk}
\usepackage{float}
\usepackage{mathrsfs}

\newtheorem*{maintheorem*}{Main Theorem}
\newtheorem{theorem}{Theorem}[section]
\newtheorem{prop}[theorem]{Proposition}

\newtheorem{lemma}[theorem]{Lemma}
\newtheorem{cor}[theorem]{Corollary}

\theoremstyle{definition}

\newtheorem{example}[theorem]{Example}
\numberwithin{equation}{section}

\newcommand{\nn}{\mathbb{N}}
\newcommand{\qq}{\mathbb{Q}}
\newcommand{\rr}{\mathbb{R}}
\newcommand{\zz}{\mathbb{Z}}

\hypersetup{
	pdftitle={FD},
	colorlinks=true,
	linkcolor=blue,
	citecolor=cyan,
	urlcolor=wine
}

\keywords{Localization, semidomain, unique factorization semidomain, half-factorial semidomain, semiring}
\subjclass[2010]{Primary: 16Y60, 13B30; Secondary: 13A05}

\begin{document}
	
	\mbox{}
	\title{Localization of unique factorization semidomains}
	
	\author{Victor Gonzalez}
	\address{Miami Dade College}
	\email{vmichelg@mdc.edu}
	
	\author{Harold Polo}
	\address{University of California, Irvine}
	\email{harold.polo@uci.edu}
	
	\author{Pedro Rodriguez}
	\address{Department of Mathematics\\Clemson University\\Clemson, SC 29634}
	\email{pedror@clemson.edu}

\date{\today}

\begin{abstract}
	A semidomain is a subsemiring of an integral domain. Within this class, a unique factorization semidomain (UFS) is characterized by the property that every nonzero, nonunit element can be factored into a product of finitely many prime elements. In this paper, we investigate the localization of semidomains, focusing specifically on UFSs. We demonstrate that the localization of a UFS remains a UFS, leading to the conclusion that a UFS is either a unique factorization domain or is additively reduced. In addition, we provide an example of a subsemiring $\mathfrak{S}$ of $\mathbb{R}$ such that $(\mathfrak{S}, \cdot)$ and $(\mathfrak{S}, +)$ are both half-factorial, shedding light on a conjecture posed by Baeth, Chapman, and Gotti.
\end{abstract}
\medskip

\maketitle


\bigskip
\section{Introduction}
\label{sec:intro}
\smallskip

The first systematic study of factorizations in integral domains was conducted in 1990 by Dan D. Anderson, David F. Anderson, and Muhammad Zafrullah \cite{AAZ90}. They introduced and analyzed various factorization properties and examined how these properties ascend from integral domains to their polynomial and power series extensions. This foundational work spurred a significant amount of subsequent research in the field \cite{AAZ92, fHK92, DDABM1996, geroldinger, mR93}. 

A subset of an integral domain containing $0$ and $1$ and closed under both addition and multiplication is called a \textit{semidomain}. One can think of a semidomain as an integral domain in which the condition of having additive inverses is no longer required. The class of semidomains provides a unified framework for studying both integral domains and positive \emph{positive semirings} (i.e., subsemirings of the positive cone of $\mathbb{R}$), which are objects that have generated much interest lately (see \cite{BCG21} and references therein). Motivated by this, the algebraic and factorization properties of semidomains have received some attention during the past few years~\cite{CF19,CCMS09, chh2022,foxgoelliao,gottipolo, gottipolo2,vP15}. Moreover, additively reduced semidomains (i.e., semidomains whose only additive inverse is $0$) provide a natural algebraic setting for studying polynomial analogues of the Goldbach conjecture as the lack of additive inverses aligns with Goldbach’s focus on positive elements~\cite{kaplan, liao}.

In this paper, we study the localization of semidomains, focusing specifically on unique factorization semidomains (UFSs) (i.e., semidomains for which every nonzero, nonunit element can be factored into a product of finitely many prime elements). We demonstrate that the localization of a UFS remains a UFS, leading to the conclusion that a UFS is either a unique factorization domain or is additively reduced. In addition, we provide an example of a subsemiring $\mathfrak{S}$ of $\mathbb{R}$ such that $(\mathfrak{S}, \cdot)$ and $(\mathfrak{S}, +)$ are both half-factorial, shedding light on a conjecture posed by Baeth, Chapman, and Gotti~\cite[Conjecture 7.7]{BCG21}.

\smallskip
\section{Preliminaries}
\label{sec:background}
\smallskip

We begin by reviewing some standard notation and terminology that will be used throughout. For further reference on factorization theory and semiring theory, consult the monographs \cite{GH06} and \cite{JG1999}, respectively. We denote the sets of positive integers, integers, rational numbers, and real numbers by $\nn$, $\zz$, $\qq$, and $\rr$, respectively, and define $\nn_0 := {0} \cup \nn$. Additionally, for any $r \in \rr$ and $S \subseteq \rr$, we define $S_{>r} \coloneqq \{s \in S \mid s > r\}$, and $S_{\geq r}$ similarly. For $m, n \in \nn_0$, we set $\llbracket m,n \rrbracket \coloneqq \{k \in \zz \mid m \leq k \leq n\}$. Given $q \in \qq_{>0}$, we can write $q$ uniquely as $q = d^{-1}n$, where $n, d \in \nn$ and $\gcd(n,d) = 1$. We refer to $n$ and $d$ as the \emph{numerator} and \emph{denominator} of $q$, denoted by $\mathsf{n}(q)$ and $\mathsf{d}(q)$, respectively.

\subsection{Monoids}

In this paper, we define a \emph{monoid}\footnote{Note that the standard definition of a monoid does not assume cancellative and commutative conditions.} as a cancellative and commutative semigroup with an identity element. For convenience, we use multiplicative notation for monoids unless specified otherwise. Let $M$ denote a monoid with identity $1$. A subsemigroup of $M$ that includes $1$ is called a \emph{submonoid}. We denote the group of units of $M$ by $M^{\times}$ and use $M_{\text{red}}$ to refer to the quotient $M/M^{\times}$, which is also a monoid. We say that $M$ is \emph{reduced} if $M^{\times}$ is the trivial group, in which case $M_{\text{red}}$ is naturally identified with $M$.

The \emph{Grothendieck group} of $M$ is an abelian group $\mathcal{G}(M)$ equipped with a monoid homomorphism $\iota \colon M \to \mathcal{G}(M)$ satisfying the following universal property: for any monoid homomorphism $f \colon M \to G$, where $G$ is an abelian group, there exists a unique group homomorphism $g \colon \mathcal{G}(M) \to G$ such that $f = g \circ \iota$. The Grothendieck group of a monoid is unique up to isomorphism. For any subset $S$ of $M$, we let $\langle S \rangle$ denote the smallest submonoid of $M$ containing~$S$, and we say that $S$ is a \emph{generating set} of~$M$ if $M = \langle S \rangle$.

For elements $b, c \in M$, we say that $b$ \emph{divides} $c$ \emph{in} $M$, denoted $b \mid_M c$, if there exists $c' \in M$ such that $c = c'b$. Two elements $b$ and $c$ are \emph{associates} in $M$, denoted $b \simeq_M c$, if $b \mid_M c$ and $c \mid_M b$. When the context makes it clear, we omit the subscript indicating the monoid. A submonoid $N$ of $M$ is called \emph{divisor-closed} if, for $b \in M$ and $c \in N$, the fact that $b \mid_M c$ implies that $b \in N$.

An \emph{atom} in a monoid $M$ is an element $a \in M \setminus M^{\times}$ such that for any $b, c \in M$, the equality $a = bc$ implies that either $b \in M^{\times}$ or $c \in M^{\times}$. We let $\mathcal{A}(M)$ represent the set of all atoms in $M$. The monoid $M$ is \emph{atomic} if every element of $M \setminus M^{\times}$ can be expressed as a finite product of atoms. It is easy to verify that $M$ is atomic if and only if $M_{\text{red}}$ is atomic.

Assuming $M$ is atomic, let $\mathsf{Z}(M)$ denote the free commutative monoid generated by the atoms in $\mathcal{A}(M_{\text{red}})$. The elements of $\mathsf{Z}(M)$ are called factorizations, and the length of a factorization $z = a_1 \cdots a_{\ell}$, where $a_1, \dots, a_{\ell} \in \mathcal{A}(M_{\text{red}})$, is denoted $|z|$. There is a unique monoid homomorphism $\pi \colon \mathsf{Z}(M) \to M_{\text{red}}$ that maps each atom to itself. For any $b \in M$, we define two sets essential to the study of factorization theory:
\[
\mathsf{Z}_M(b) = \pi^{-1}(b\mathcal{U}(M)) \subseteq \mathsf{Z}(M) \quad \text{and} \quad \mathsf{L}_M(b) = \{|z| : z \in \mathsf{Z}_M(b)\} \subseteq \nn_0.
\]
As is standard, we drop the subscript $M$ when no ambiguity arises. We use the following terminology: $M$ is a \emph{half-factorial monoid} if $|\mathsf{L}(b)| = 1$ for all $b \in M$, and a \emph{unique factorization monoid} if $|\mathsf{Z}(b)| = 1$ for all $b \in M$. A unique factorization monoid is also referred to as \emph{factorial}. It follows directly from the definitions that every unique factorization monoid is half-factorial.

\subsection{Semidomains}

A \emph{commutative semiring} $\mathfrak{S}$ is a non-empty set equipped with two binary operations: addition and multiplication, denoted by ``$+$'' and ``$\cdot$'', respectively. Specifically, a commutative semiring satisfies the following properties:

\begin{enumerate}
	\item $(\mathfrak{S}, +)$ is a monoid with an identity element $0$;
	\item $(\mathfrak{S}, \cdot)$ is a commutative semigroup with an identity element $1$, where $1 \neq 0$; and
	\item  $s_1 \cdot (s_2 + s_3) = s_1 \cdot s_2 + s_1 \cdot s_3$ for all $s_1, s_2, s_3 \in \mathfrak{S}$.
\end{enumerate}

For a commutative semiring $\mathfrak{S}$, the distributive law and the cancellative property of addition ensure that $0 \cdot s = 0$ for all $s \in \mathfrak{S}$. Throughout this paper, we denote products of elements in $\mathfrak{S}$ by $ss'$ rather than $s \cdot s'$. It is worth noting that in the general definition of a semiring, commutativity in multiplication and the cancellative property of addition are not required. However, we are concerned with semirings that are both commutative and additive-cancellative. For simplicity, we will refer to these simply as \emph{semirings}, assuming commutativity and cancellativity for both operations. A \emph{semiring homomorphism} is a map between two semirings that preserves both the additive and multiplicative structures. Formally, let $(R, +, \cdot, 0, 1)$ and $(S, \oplus, \odot, 0_S, 1_S)$ be semirings. A function $\varphi : R \to S$ is a semiring homomorphism if it satisfies the following conditions for all $a, b \in R$:
\begin{enumerate}
	\item $\varphi(a + b) = \varphi(a) \oplus \varphi(b)$;
	\item $\varphi(a \cdot b) = \varphi(a) \odot \varphi(b)$;
	\item $\varphi(0) = 0_S$;
	\item $\varphi(1) = 1_S$.
\end{enumerate}

An \emph{ideal} $I$ of a semidomain $\mathfrak{S}$ is a nonempty subset of $\mathfrak{S}$ satisfying the following conditions:
\begin{enumerate}
	\item if $i,i' \in I$ then $i + i' \in I$;
	\item if $s \in \mathfrak{S}$ and $i \in I$ then $si \in I$.
\end{enumerate}
We say that $I$ is a \emph{proper} ideal if $I \subsetneq \mathfrak{S}$. On the other hand, an ideal $I$ is called \emph{subtractive} if, for $s \in \mathfrak{S}$ and $i \in I$, the condition $s + i \in I$ implies that $s \in I$. Observe that all ideals of a commutative ring are subtractive. However, the ideals of a semidomain need not be subtractive. Indeed, the ideal $I = \nn_0 \setminus \{1\}$ of $\nn_0$ is clearly not subtractive. A proper ideal $I$ is called \emph{prime} if whenever $ss' \in I$ for some $s,s' \in \mathfrak{S}$, we have that either $s \in I$ or $s' \in I$. On the other hand, the ideal $I$ is \emph{maximal} if there is no ideal $I' \subsetneq  S$ such that $I \subsetneq I'$. 

A subset $\mathfrak{S}' \subseteq \mathfrak{S}$ is called a \emph{subsemiring} if $(\mathfrak{S}', +)$ is a submonoid of $(\mathfrak{S}, +)$, closed under multiplication, and contains the multiplicative identity $1$. A semiring $\mathfrak{S}$ is called a \emph{semidomain} if it is a subsemiring of an integral domain. For a semidomain $\mathfrak{S}$, the set $\mathfrak{S} \setminus \{0\}$ forms a monoid under multiplication, denoted by $\mathfrak{S}^*$, and referred to as the \emph{multiplicative monoid of $\mathfrak{S}$}. To distinguish between additive and multiplicative units, we refer to the units of $(\mathfrak{S}, +)$ as \emph{invertible elements} and denote them by $\mathcal{U}(\mathfrak{S})$, while we denote the units of $\mathfrak{S}^*$ as $\mathfrak{S}^{\times}$. We use $\mathcal{A}(\mathfrak{S})$ to denote the set of atoms of the multiplicative monoid $\mathfrak{S}^*$, and we write $s \mid_\mathfrak{S} s'$ to indicate that $s$ divides $s'$ in $\mathfrak{S}^*$.

\begin{lemma} \label{lem:characterization of integral semirings}
	For a semiring $\mathfrak{S}$, the following conditions are equivalent.
	\begin{enumerate}
		\item[(a)] $\mathfrak{S}$ is a semidomain.
		\item[(b)] The multiplication of $\mathfrak{S}$ extends to the Grothendieck group $\mathcal{G}(\mathfrak{S})$ of $(\mathfrak{S}, +)$, making $\mathcal{G}(\mathfrak{S})$ an integral domain.
	\end{enumerate}
\end{lemma}

\begin{proof}
	Proving that statement (b) implies statement (a) is straightforward. So, we leave this task to the reader. Now suppose $\mathfrak{S}$ is a semidomain embedded in an integral domain $R$. We can identify the Grothendieck group $\mathcal{G}(\mathfrak{S})$ of $(\mathfrak{S}, +)$ with the subset $\{r - s \mid r, s \in \mathfrak{S}\}$ of the underlying additive group of $R$. This group is closed under multiplication inherited from $R$ and includes the multiplicative identity since $0, 1 \in \mathfrak{S}$. Thus, $\mathcal{G}(\mathfrak{S})$ is an integral domain containing $S$ as a subsemiring.
\end{proof}

A semidomain $\mathfrak{S}$ is \emph{atomic} if $\mathfrak{S}^*$ is an atomic monoid. If $\mathfrak{S}^*$ is a half-factorial monoid, then $\mathfrak{S}$ is called a \emph{half-factorial semidomain} (HFS). Likewise, if $\mathfrak{S}^*$ is a unique factorization monoid, $\mathfrak{S}$ is called a \emph{unique factorization semidomain} (UFS). When $\mathfrak{S}$ is an integral domain, this recovers the usual concept of a UFD. Note, however, that while a semidomain $\mathfrak{S}$ can embed into an integral domain $R$, it may not inherit the atomic properties of $R$. For instance, the integral domain $\mathbb{Q}[x]$ is a UFD, but the subring $\mathbb{Z} + x\mathbb{Q}[x]$ is not atomic.

\smallskip
\section{Localization of Semidomains}
\smallskip

Localization is a fundamental construction in commutative algebra, traditionally developed in the context of commutative rings (or modules). In this section, we explore its adaptation to semidomains, showing that many familiar properties and techniques carry over naturally. For a more general treatment of localization in similar settings, see \cite[Chapter 11]{JG1999}.

Let $\mathfrak{S}$ be a semidomain, and let $S$ be a \emph{multiplicative subset} of $\mathfrak{S}$ (i.e., a submonoid of $\mathfrak{S}^*$). Since $S$ is also a multiplicative subset of the integral domain $\mathcal{G}(\mathfrak{S})$, we can consider the localization of $\mathcal{G}(\mathfrak{S})$ at $S$, which we denote by $S^{-1}\mathcal{G}(\mathfrak{S})$. Now set $R \coloneqq (\mathfrak{S} \times S)/\sim$, where $\sim$ is an equivalence relation on $\mathfrak{S} \times S$ defined by $(s,d) \sim (s',d')$ if and only if $sd' = ds'$. We let $\frac{s}{d}$ denote the equivalence class of $(s,d)$. Define the following operations in $R$:
\[
\frac{s}{d} \cdot \frac{s'}{d'} = \frac{ss'}{dd'} \hspace{.3 cm} \text{ and } \hspace{.3 cm} \frac{s}{d} + \frac{s'}{d'} = \frac{sd' + ds'}{dd'}.
\]
It is routine to verify that these operations are well defined and that $(R,+,\cdot)$ is a semiring. So, we leave the details to the reader. The semiring $R$ is called \emph{the localization of $\mathfrak{S}$ at $S$} and is denoted by $S^{-1}\mathfrak{S}$. Let $\varphi\colon R \to S^{-1}\mathcal{G}(\mathfrak{S})$ be a map given by $\varphi(\frac sd) = \overline{\frac sd}$, where $\overline{\frac sd}$ represents the equivalence class of $(s,d)$ as an element of $S^{-1}\mathcal{G}(\mathfrak{S})$. Note that $\varphi$ is a well-defined semiring homomorphism. Since $\varphi$ is injective, we can conclude that $R$ is, in fact, a semidomain. 

As with integral domains, the localization of a semidomain can be characterized up to isomorphism as demonstrated by the following result.

\begin{prop} \label{prop: map from a semidomain to its localization}
	Let $\mathfrak{S}$ be a semidomain, and let $S$ be a multiplicative subset of $\mathfrak{S}$. The following statements hold.
	\begin{enumerate}
		\item The map $\pi\colon \mathfrak{S} \to S^{-1}\mathfrak{S}$ given by $\pi(t) = \frac{t}{1}$ is an injective semiring homomorphism satisfying that $\pi(s)$ is a unit in $S^{-1}\mathfrak{S}$ for all $s \in S$.
		\item If $\theta\colon \mathfrak{S} \to \mathfrak{S}'$ is a semiring homomorphism such that $\theta(s)$ is a unit in $\mathfrak{S}'$ for every $s \in S$ then there exists a unique semiring homomorphism $\phi\colon S^{-1}\mathfrak{S} \to \mathfrak{S}'$ such that $\phi \circ \pi = \theta$.
	\end{enumerate}
\end{prop}

\begin{proof}
	Proving the statement $(1)$ is straightforward, so we leave the details to the reader. Now consider the semiring homomorphism $\phi\colon S^{-1}\mathfrak{S} \to \mathfrak{S}'$ such that $\phi(\frac{t}{s}) = \theta(t)\theta(s)^{-1}$. Note that, for every $t \in \mathfrak{S}$, we have
	\[
	(\phi \circ \pi)(t) = \phi(\pi(t)) = \phi\left(\frac{t}{1}\right) = \theta(t).
	\]
	Suppose that there exists $\phi'\colon S^{-1}\mathfrak{S} \to \mathfrak{S}'$ such that $\phi' \circ \pi = \theta$. Thus, for every $\frac{t}{s} \in S^{-1}\mathfrak{S}$, we have
	\[
	\phi'\left(\frac{t}{s}\right) = \phi'\left(\frac{t}{1}\right)\phi'\left(\frac{1}{s}\right) = (\phi' \circ \pi)(t)(\phi' \circ \pi)(s)^{-1} = \theta(t)\theta(s)^{-1} = \phi\left(\frac{t}{s}\right).
	\]
	Therefore, the map $\phi$ is unique, and $(2)$ follows.
\end{proof}

The localization of a semidomain $\mathfrak{S}$ with respect to two different multiplicative subsets, $S_1$ and $S_2$, can sometimes result in isomorphic semidomains, $S_1^{-1}\mathfrak{S}$ and $S_2^{-1}\mathfrak{S}$. The following corollary provides further insight into this phenomenon, but we first need a couple of definitions. A multiplicative subset $S$ of $\mathfrak{S}$ is \emph{saturated} if for $s \in \mathfrak{S}$ and $s' \in S$, we have that $s \mid_{\mathfrak{S}} s'$ implies that $s \in S$. Given a multiplicative subset $S$, we define its saturation as $\overline{S} \coloneqq \{s \in \mathfrak{S} : s \mid_{\mathfrak{S}} s' \text{ for some } s' \in S\}$. It is not hard to see that $\overline{S}$ is a saturated subset of $\mathfrak{S}$.   

\begin{cor} \label{cor: saturated}
	Let $\mathfrak{S}$ be a semidomain, and let $S$ be a multiplicative subset of $\mathfrak{S}$. Then $S^{-1}\mathfrak{S}$ and $\overline{S}^{-1}\mathfrak{S}$ are isomorphic.
\end{cor}

\begin{proof}
	Let $\pi_1\colon \mathfrak{S} \to S^{-1}\mathfrak{S}$ be a map given by $\pi_1(s) = \frac{s}{1}$. Similarly, let $\pi_2\colon \mathfrak{S} \to \overline{S}^{-1}\mathfrak{S}$ be a map given by $\pi_2(s) = \frac{s}{1}$. Since $S \subseteq \overline{S}$, there exists a unique semiring homomorphism $\phi_1 \colon S^{-1}\mathfrak{S} \to \overline{S}^{-1}\mathfrak{S}$ such that $\phi_1 \circ \pi_1 = \pi_2$. Moreover, observe that if $s \mid_{\mathfrak{S}} s'$ for some $s' \in S$, then $\pi_1(s)$ is a unit of $S^{-1}\mathfrak{S}$. Consequently, there exists a unique semiring homomorphism $\phi_2\colon \overline{S}^{-1}\mathfrak{S} \to S^{-1}\mathfrak{S}$ such that $\phi_2 \circ \pi_2 = \pi_1$. Therefore, the map $\phi_1$ is a semiring homomorphism.
\end{proof}

Given an ideal $I$ of a semidomain $\mathfrak{S}$, we denote by $S^{-1}I$ the ideal of the localization $S^{-1}\mathfrak{S}$ generated by the set $\pi(I)$, where the map $\pi$ is defined as in Proposition~\ref{prop: map from a semidomain to its localization}. We say that the ideal $S^{-1}I$ is the \emph{extension of $I$ by $\pi$}.

\begin{prop} \label{prop: bijection between ideals of a semidomain and its localization}
	Let $\mathfrak{S}$ be a semidomain, and let $S$ be a multiplicative subset of $\mathfrak{S}$. The following statements hold.
	\begin{enumerate}
		\item For any ideal $J$ of $S^{-1}\mathfrak{S}$ we have that $S^{-1}\pi^{-1}(J) = J$. In other words, every ideal of $S^{-1}\mathfrak{S}$ is the extension of an ideal in $\mathfrak{S}$.
		\item For an ideal $I$ of $\mathfrak{S}$, we have that $S^{-1}I = S^{-1}\mathfrak{S}$ if and only if $I \cap S \neq \emptyset$.
		\item The assignment $I \mapsto S^{-1}I$ induces a bijection between the set of prime ideals of $\mathfrak{S}$ disjoint from $S$ and the set of prime ideals of $S^{-1}\mathfrak{S}$.
	\end{enumerate}
\end{prop}

\begin{proof}
	Set $J' \coloneqq S^{-1}\pi^{-1}(J)$, and observe that evidently $J' \subseteq J$. Let $j \in J$, and write it as $j = \frac{t}{s}$ with $t \in \mathfrak{S}$ and $s \in S$. Since $\frac{s}{1}\cdot \frac{t}{s} \in J$, we have that $t \in \pi^{-1}(J)$ which, in turn, implies that $j \in J'$. Hence the statement $(1)$ follows.
	
	Suppose now that $S^{-1}I = S^{-1}\mathfrak{S}$ for some ideal $I$ of $\mathfrak{S}$. Then there exists $i \in I$ and $s \in S$ such that $\frac{i}{s} = \frac{1}{1}$. This implies that $i \in I \cap S$. Conversely, let $i \in I \cap S$. In this case, we have that $\frac{i}{i} \in S^{-1}I$. Consequently, $S^{-1}I = S^{-1}\mathfrak{S}$.
	
	Assume that $I$ is a prime ideal of $\mathfrak{S}$ disjoint from $S$. We claim that $S^{-1}I$ is a prime ideal of $S^{-1}\mathfrak{S}$. Note that $S^{-1}I$ is a proper ideal of $S^{-1}\mathfrak{S}$ by $(2)$. Now let $\frac{s}{t}, \frac{s'}{t'} \in S^{-1}\mathfrak{S}$ such that $\frac{ss'}{tt'} \in S^{-1}I$. This implies that $\frac{ss'}{tt'} = \frac{i}{t''}$ for some $i \in I$ and $t'' \in S$. Since $I$ is disjoint from $S$, we have that either $s \in I$ or $s' \in I$, which concludes the proof of our claim. On the other hand, it is not hard to see that if $J$ is a prime ideal of $S^{-1}\mathfrak{S}$ then $\pi^{-1}(J)$ is a prime ideal of $\mathfrak{S}$, which is disjoint from $S$ by condition $(2)$. Then statement $(3)$ follows from statement $(1)$. 
\end{proof}

\smallskip
\section{Localizations of Unique Factorization Semidomains}
\smallskip

In the context of integral domains, it is well known that the localization of a UFD is a UFD. It is also known that if an integral domain $\mathfrak{S}$ is atomic and a multiplicative closed subset $S$ is generated by primes, then the localization $S^{-1}\mathfrak{S}$ being a UFD implies that $S$ is also a UFD. In this section, we extend these results to the more general class of semidomains. To do so, we first provide a characterization of UFSs in terms of their prime ideals. Specifically, we extend Kaplansky's characterization of UFDs \cite[Theorem~5]{iK74} to the context of semidomains.

\begin{lemma}\label{prop: existence of prime ideals disjoint from a multiplicatively closed subset}
	Let $\mathfrak{S}$ be a semidomain, and let $S$ be a multiplicative subset of $\mathfrak{S}$. If there exists an ideal $I$ satisfying that $S \cap I = \emptyset$ then there exists a prime ideal $P$ such that $S \cap P = \emptyset$.
\end{lemma}

\begin{proof}
	Let $\mathcal{I}$ be the set consisting of all ideals of $\mathfrak{S}$ that contain $I$ and are disjoint from $S$. Observe that $\mathcal{I}$ is nonempty as $I \in \mathcal{I}$. Take an arbitrary chain $\{I_{\alpha} \mid \alpha \in \Delta\}$ in $\mathcal{I}$, and consider the ideal $J= \cup_{\alpha \in \Delta} I_\alpha$. Clearly, the ideal $J$ contains $I$ and the equality $S \cap J = \emptyset$ holds. In other words, every chain in $\mathcal{I}$ has an upper bound. By Zorn's lemma, the set $\mathcal{I}$ contains a maximal element $P$. We now prove that $P$ is prime. Let $s, s' \in \mathfrak{S}$ such that $ss' \in P$. Suppose towards a contradiction that $s \notin P$ and $s' \notin P$. Since both $P + (s)$ and $P + (s')$ properly contain $P$, we obtain that $P + (s)$ and $P + (s')$ are not elements of $\mathcal{I}$. Consequently, there exist $s_1 \in S \cap (P + (s))$ and $s_2 \in S \cap (P + (s'))$. Now write $s_1 = p_1 + s_3s$ and $s_2 = p_2 + s_4s'$ for some $p_1, p_2 \in P$ and $s_3, s_4 \in \mathfrak{S}$. Hence $s_1s_2 \in S \cap P$, which contradicts that $P \in \mathcal{I}$. Therefore, the ideal $P$ is prime.
\end{proof}

Now we are in a position to extend Kaplansky's characterization of UFDs \cite[Theorem~5]{iK74} to the context of semidomains.

\begin{prop}\label{Kaplansky for semidomains}
	A semidomain $\mathfrak{S}$ is a UFS if and only if every nonzero prime ideal of $\mathfrak{S}$ contains a prime element.  
\end{prop}  

\begin{proof}
	Suppose that $\mathfrak{S}$ is a UFS. Without loss of generality, assume that $S$ is not a semifield. Let $P$ be a nonzero prime ideal of $\mathfrak{S}$, and take $x \in P\setminus\{0\}$. Since $x$ is not a unit, we can write $x = p_1 \cdots p_n$ for prime elements $p_1, \ldots, p_n \in \mathfrak{S}$. Hence there exists $i \in \llbracket 1,n \rrbracket$ such that $p_i \in P$ because $P$ is a prime ideal.
	
	Now assume that $\mathfrak{S}$ is a semidomain for which every nonzero prime ideal contains a prime element. Suppose towards a contradiction that $S$ is not a UFS. Let $S$ be the set of all elements of $\mathfrak{S}$ that can be expressed as a finite product of units and prime elements. It is not hard to see that $S$ is a divisor-closed submonoid of $\mathfrak{S}^*$. Since $\mathfrak{S}$ is not a UFS, there exists $s \in \mathfrak{S}$ such that $s \notin S$. Hence the principal ideal $(s)$ is disjoint from $S$. By Lemma~\ref{prop: existence of prime ideals disjoint from a multiplicatively closed subset}, we obtain that there exists a prime ideal $P$ such that $(x) \subseteq P$ and $P \cap S = \emptyset$. However, this contradicts the assumption that every nonzero prime ideal contains a prime element, which concludes our argument. 
\end{proof}

We can now prove the main result of this section.

\begin{theorem} \label{UFS}
	Let $\mathfrak{S}$ be a semidomain, and let $S$ be a multiplicative subset of $\mathfrak{S}$. The following statements hold. 
	\begin{enumerate}
		\item If $\mathfrak{S}$ is a UFS then $S^{-1}\mathfrak{S}$ is a UFS.
		\item Suppose that $S$ is generated by primes and $\mathfrak{S}$ is atomic. If $S^{-1}\mathfrak{S}$ is a UFS then $\mathfrak{S}$ is a UFS.
	\end{enumerate}
\end{theorem}

\begin{proof}
	Suppose that $\mathfrak{S}$ is a UFS and that $S$ is a multiplicative subset of $\mathfrak{S}$. Let $P$ be an arbitrary prime ideal of $S^{-1}\mathfrak{S}$. By Proposition~\ref{prop: bijection between ideals of a semidomain and its localization}, we know that $\pi^{-1}(P)$ is a prime ideal of $\mathfrak{S}$ that is disjoint from $S$, where $\pi$ is defined as in Proposition~\ref{prop: map from a semidomain to its localization}. Now, by Proposition~\ref{Kaplansky for semidomains}, there exists a prime element $p \in \pi^{-1}(P)$ (thus $p \notin S$). We claim that $\frac{p}{s}$ is a prime element of $S^{-1}\mathfrak{S}$. Suppose towards a contradiction that
	\[
		\frac{p}{s} \;\;\bigg|_{S^{-1}\mathfrak{S}}\; \frac{\mathfrak{s}_1}{s_1}\cdot\frac{\mathfrak{s}_2}{s_2}, \hspace{.3 cm} \frac{p}{s} \;\;\not\bigg|_{S^{-1}\mathfrak{S}}\;\; \frac{\mathfrak{s}_1}{s_1}, \hspace{.3 cm}\text{ and }\hspace{.3 cm} \frac{p}{s} \;\;\not\bigg|_{S^{-1}\mathfrak{S}}\;\; \frac{\mathfrak{s}_2}{s_2}
	\]
	for some $\mathfrak{s}_1, \mathfrak{s}_2 \in \mathfrak{S}$ and $s_1, s_2 \in S$. Then there exist $\mathfrak{s}_3 \in \mathfrak{S}$ and $s_3 \in S$ such that either $p \;|_{\mathfrak{S}}\; ss_3$ or $p \;|_{\mathfrak{S}}\; \mathfrak{s}_1\mathfrak{s}_2$. By virtue of Corollary~\ref{cor: saturated}, there is no loss in assuming that $S$ is saturated. This implies that either $p \;|_{\mathfrak{S}}\; \mathfrak{s}_1$ or $p \;|_{\mathfrak{S}}\; \mathfrak{s}_2$. From this contradiction, our claim readily follows. Therefore, the semidomain $S^{-1}\mathfrak{S}$ is a UFS by Proposition~\ref{Kaplansky for semidomains}, which concludes the proof of statement $(1)$. 
	
	Now suppose that $S$ is generated by primes, the semidomain $\mathfrak{S}$ is atomic, and $S^{-1}\mathfrak{S}$ is a UFS. By Proposition~\ref{Kaplansky for semidomains}, every nonzero prime ideal of $S^{-1}\mathfrak{S}$ contains a prime element. Let $P$ be a nonzero prime ideal of $\mathfrak{S}$ and, without loss of generality, assume that $S \cap P = \emptyset$. By our assumption, the prime ideal $S^{-1}P$ contains a prime element $\frac{p}{s}$. Since $\mathfrak{S}$ is atomic and $p \notin S$, there is no loss in assuming that $p = a_1 \cdots a_n$, where $a_i \in \mathcal{A}(\mathfrak{S}) \setminus S$ for every $i \in \llbracket 1,n \rrbracket$. Suppose towards a contradiction that $p$ is not a prime element of $\mathfrak{S}$. Then there exist $\mathfrak{s}_1, \mathfrak{s}_2 \in \mathfrak{S}$ such that $p \;|_{\mathfrak{S}}\; \mathfrak{s}_1\mathfrak{s}_2$ but $p \not|_{\mathfrak{S}}\, \mathfrak{s}_1$ and $p \not|_{\mathfrak{S}}\, \mathfrak{s}_2$. Since $\frac{p}{s} \,|_{S^{-1}\mathfrak{S}}\, \frac{\mathfrak{s}_1}{s} \cdot \frac{\mathfrak{s}_2}{s}$, we may assume that $\frac{p}{s}\,|_{S^{-1}\mathfrak{S}}\, \frac{\mathfrak{s}_1}{s}$. This implies that $p\mathfrak{s}_3 = s'\mathfrak{s}_1$ for some $\mathfrak{s}_3 \in \mathfrak{S}$ and $s' \in S$. Since $S$ is generated by primes, we have that $s' = p_1 \cdots p_m$ for prime elements $p_1, \ldots, p_m \in \mathfrak{S}$. Since $p_i$ is not an associate of $a_j$ for any $i \in \llbracket 1,m \rrbracket$ and any $j \in \llbracket 1,n \rrbracket$, we have that $p \,|_{\mathfrak{S}}\, \mathfrak{s}_1$, which is a contradiction. By Proposition~\ref{Kaplansky for semidomains}, we have that $\mathfrak{S}$ is a UFS.
\end{proof}

A semidomain $\mathfrak{S}$ is called \emph{semisubtractive} if, for every $s \in \mathcal{G}(\mathfrak{S})$, we have that either $s \in \mathfrak{S}$ or $-s \in \mathfrak{S}$. Fox et al.~\cite{foxgoelliao} showed that semisubtractive semidomains $\mathfrak{S}$ share many algebraic and factorization properties with their domains of differences $\mathcal{G}(\mathfrak{S})$. In particular, they established that if a semidomain $\mathfrak{S}$ that is not an integral domain satisfies the unique factorization property then $\mathfrak{S}$ is additively reduced. We now extend this result to the class of all semidomains.

\begin{lemma}\label{lemma: localizing commutes with taking polynomial extension}
	Let $\mathfrak{S}$ be a semidomain, and let $S$ be a multiplicative subset of $\mathfrak{S}$. Then $S^{-1}(\mathfrak{S}[x]) \cong (S^{-1}\mathfrak{S})[x]$.
\end{lemma}

\begin{proof}
	Observe that $S$ is also a multiplicative subset of $\mathfrak{S}[x]$, so the construction $S^{-1}(\mathfrak{S}[x])$ makes total sense. Now let us define the map $\phi \colon S^{-1}(\mathfrak{S}[x]) \rightarrow  (S^{-1}\mathfrak{S})[x]$ given by
	\[
		\phi\left(\frac{\sum_{i = 0}^n c_ix^i}{s}\right) = \sum_{i = 0}^{n} \frac{c_i}{s}x^i,
	\]
	where the coefficients $c_0, \ldots, c_n \in \mathfrak{S}$. It is not hard to see that the function $\phi$ is well defined. Verifying that $\phi$ is a semiring homomorphism is straightforward but tedious, so we leave the proof to the reader. Now let $f = \sum_{i = 0}^{n}\frac{d_i}{s_i}x^i$ be an arbitrary element of $(S^{-1}\mathfrak{S})[x]$. For each $j \in \llbracket 0,n \rrbracket$, set $s_j^* \coloneqq \frac{s_0 \cdots s_n}{s_j}$. Thus, 
	\[
		\phi\left(\frac{\sum_{i = 0}^{n} d_is_i^*x^i}{s_0 \cdots s_n}\right) = f,
	\]  
	which implies that the map $\phi$ is surjective. Supose that there exist $f,g \in \mathfrak{S}[x]$ and $s,s' \in S$ such that $\phi(\frac{f}{s}) = \phi(\frac{g}{s'})$. Write $f = \sum_{i = 0}^{n} c_ix^i$ and $g = \sum_{i = 0}^{n} d_ix^i$, where the coefficients $c_i$ and $d_i$ are elements of $\mathfrak{S}$ for every $i \in \llbracket 0,n \rrbracket$. Since $\phi(\frac{f}{s}) = \phi(\frac{g}{s'})$, we obtain that $\frac{c_i}{s} = \frac{d_i}{s'}$ for every $i \in \llbracket 0,n \rrbracket$. This, in turn, implies that $\frac{f}{s} = \frac{g}{s'}$. We can then conclude that the map $\phi$ is a semiring isomorphism.
\end{proof}

Next we show that if $\mathfrak{S}$ is a UFS then either $\mathfrak{S}$ is a UFD or $\mathfrak{S}$ is additively reduced, which sheds light upon \cite[Question~7.7]{BCG21}.

\begin{prop} \label{prop: UFS are either UFDs or additively reduced}
	Let $\mathfrak{S}$ be a UFS. Then either $\mathfrak{S}$ is an integral domain or it is additively reduced.
\end{prop}

\begin{proof}
	By way of contradiction, assume that $\mathfrak{S}$ is neither an integral domain nor an additively reduced semidomain. Then there exists $s_0 \in \mathfrak{S}\setminus \{0\}$ such that $-s_0 \in \mathfrak{S} \setminus \{0\}$. Set $S \coloneqq \mathfrak{S} \setminus \{0\}$, and consider the semifield $S^{-1}\mathfrak{S}$. For $\frac{t}{s} \in S^{-1}\mathfrak{S}$, we have that $-\frac{t}{s} = \frac{t}{s} \cdot \frac{-s_0}{s_0}$, which implies that $S^{-1}\mathfrak{S}$ is indeed a field. Since $S^{-1}(\mathfrak{S}[x]) = (S^{-1}\mathfrak{S})[x]$ by Lemma~\ref{lemma: localizing commutes with taking polynomial extension}, we have that $S^{-1}(\mathfrak{S}[x])$ is a UFD. 
	
	Let $p$ be a prime element of $\mathfrak{S}$. We shall prove that $p$ is also a prime element of $\mathfrak{S}[x]$. Suppose towards a contradiction that there exist $f, g \in \mathfrak{S}[x]$ such that $p$ divides $f g$ in $\mathfrak{S}[x]$ and $p$ does not divide $f$ or $g$. Write
	\[
	f = s_nx^{m_n} + \cdots + s_0x^{m_0} \hspace{.5 cm} \text{ and } \hspace{.6 cm} g = s_t'x^{k_t} + \cdots + s_0'x^{k_0}
	\]
	for exponents $m_n, \ldots, m_0, k_t, \ldots, k_0 \in \nn_0$ and coefficients $s_n, \ldots, s_0,s_t', \ldots, s_0' \in S^*$. There is no loss in assuming that $m_n > \cdots > m_0$ and $k_t > \cdots > k_0$. Let $i, j$ be the largest indices such that $p \nmid_{\mathfrak{S}} s_i$ and $p \nmid_{\mathfrak{S}} s_j'$. Since $p$ divides $fg$ in $\mathfrak{S}[x]$, we have that $p$ divides 
	\[
	s_is_j' + \sum_{\substack{m_{\ell} + k_r = m_i + k_j}} s_\ell s_r' 
	\]
	in $\mathfrak{S}$, where either $\ell > i$ or $r > j$. By the maximality of the indices $i$ and $j$, we obtain that $p \mid_{\mathfrak{S}} s_is_j'$, which contradicts that $p$ is a prime element of $\mathfrak{S}$. Hence $S$ is a multiplicative subset of $\mathfrak{S}[x]$ generated by primes. Moreover, observe that the semidomain $\mathfrak{S}[x]$ is atomic by \cite[Theorem~4.1]{gottipolo}. By Theorem~\ref{UFS}, we can conclude that $\mathfrak{S}[x]$ is a UFS, which contradicts \cite[Theorem~6.6]{gottipolo}. Therefore, $\mathfrak{S}$ is either an integral domain or an additively reduced semidomain. 
\end{proof}

Following \cite{BCG21}, we say that a semidomain $\mathfrak{S}$ is \emph{bi-HFS} if both monoids $(\mathfrak{S},+)$ and $\mathfrak{S}^*$ are HFMs. In \cite{BCG21}, Baeth et al. posed the following question: Is $\nn_0$ the only subsemiring of $\rr_{\geq 0}$ that is bi-HFS? We now show that there exist a subsemiring of $\rr$ different from $\nn_0$ that is bi-HFS.
 
 \begin{example}
 	Let $\mathfrak{S} = \{0\} \cup \{f \in \zz[x] \mid f(0) > 0\}$. Clearly, $\mathfrak{S}$ is a semidomain satisfying that $\mathfrak{S}^{\times} = \{1\}$. Suppose towards a contradiction that there exists an irreducible $f \in \mathfrak{S}$ that is not irreducible in $\zz[x]$. Write $f = f_1f_2$ for non-unit elements $f_1, f_2 \in \zz[x]$. We may assume that $f_1 \notin S$, i.e., $f_1(0) < 0$. This implies that $f_2(0) < 0$. Consequently, we can write $f = (-f_1)(-f_2)$ with $-f_1, -f_2 \in \mathfrak{S}$. This contradiction shows that $\mathcal{A}(\mathfrak{S}) \subseteq \mathcal{A}(\zz[x])$. Since $\mathfrak{S}$ is atomic, we can conclude that $S^*$ is half-factorial. On the other hand, it is not hard to see that $\mathcal{A}_{+}(\mathfrak{S}) = \{f \in \mathfrak{S} \mid f(0) = 1\}$ which, in turn, implies that $(\mathfrak{S},+)$ is atomic. Note also that the length of an additive factorization of an element $g \in \mathfrak{S}$ is equal to $g(0)$. Hence $\mathfrak{S}$ is a bi-HFS.
 \end{example}

\bigskip
\section*{Acknowledgments}

During the preparation of this paper, the second author gratefully acknowledges the support of the University of California President's Postdoctoral Fellowship.

\bigskip

\end{document}